\documentclass[a4paper,12pt]{article}
\usepackage{a4wide}
\usepackage{amsmath}
\usepackage{amssymb}
\usepackage{amsthm}
\usepackage{latexsym}
\usepackage{graphicx}
\usepackage[english]{babel}
\usepackage{makeidx}
              
\usepackage{amsmath,amsthm}
\usepackage{amssymb}

\usepackage{enumitem}

\usepackage{graphicx}

\usepackage[T1]{fontenc}

\usepackage{a4wide}
\usepackage{amsmath}
\usepackage{amssymb}
\usepackage{amsthm}
\usepackage{latexsym}
\usepackage{graphicx}
\usepackage[english]{babel}
\usepackage{makeidx}
\usepackage{mathtools}

\newtheorem{obs} [subsection]{Remark}
\newtheorem{exm} [subsection]{Example}

\newtheorem{prop}[subsection]{Proposition}

\newtheorem{teor}[subsection]{Theorem}

\newtheorem{cor} [subsection]{Corollary}

\def\supp{\operatorname{supp}}

\def\Ker{\operatorname{Ker}}

\def\ord{\operatorname{ord}}
\numberwithin{equation}{section}

\theoremstyle{definition}

\numberwithin{equation}{section}

\begin{document}
\selectlanguage{english}
\frenchspacing
\numberwithin{equation}{section}




\title{Two semigroup rings associated to a finite set of germs of meromorphic functions}

\author{Mircea Cimpoea\c s}

\maketitle


\begin{abstract}
We fix $z_0\in\mathbb C$ and a field $\mathbb F$ with $\mathbb C\subset \mathbb F \subset \mathcal M_{z_0}:=$ the field of germs of meromorphic functions at $z_0$.
We fix $f_1,\ldots,f_r\in \mathcal M_{z_0}$ and we consider the $\mathbb F$-algebras $S:=\mathbb F[f_1,\ldots,f_r]$ and $\overline S:=\mathbb F[f_1^{\pm 1},\ldots,f_r^{\pm 1}]$.
We present the general properties of the semigroup rings
\begin{align*}
& S^{hol}:=\mathbb F[f^{\mathbf a}:=f_1^{a_1}\cdots f_r^{a_r}: (a_1,\ldots,a_r)\in\mathbb N^r \text{ and }f^{\mathbf a}\text{ is holomorphic at }z_0],\\
& \overline S^{hol}:=\mathbb F[f^{\mathbf a}:=f_1^{a_1}\cdots f_r^{a_r}: (a_1,\ldots,a_r)\in\mathbb Z^r \text{ and }f^{\mathbf a}\text{ is holomorphic at }z_0],
\end{align*}
and we tackle in detail the case in which $\mathbb F=\mathcal M_{<1}$ is the field of meromorphic functions of order $<1$ and 
$f_j$'s are meromorphic functions over $\mathbb C$ of finite order with a finite number of zeros and poles.


\end{abstract}

\section{Introduction}

Let $z_0\in \mathbb C$ and let $g$ be a holomorphic function at $z_0$, that is $g$ is holomorphic on an open domain $U\subset \mathbb C$ with $z_0\in U$.
Replacing $g(z)$ with $g(z-z_0)$, we can assume that $z_0=0$. 
Given two holomorphic functions $g_1$ and $g_2$ at $0$ we say that $g_1\sim g_2$ if there exist an open domain $U\ni z_0$ such that $g_1|_U = g_2|_U$.
$\sim$ is an equivalence relation. A class of equivalence of $\sim$ is called a germ of holomorphic function. 
We denote $\mathcal O_{0}$ the ring of germs of holomorphic functions at $0$. 
It is well known that 
$$\mathcal O_{0} \cong \mathbb C\{z\} = \{\sum_{n=0}^{+\infty} a_nz^n\;:\; \frac{1}{\limsup_n \sqrt[n]{|a_n|}} >0\},$$
the ring of convergent power series, 
which is an one dimensional local regular ring with the maximal ideal $\mathfrak m = z\mathbb C\{z\}$.

Let $f$ be a meromorphic function at $0$, that is there exists an open domain $U\subset \mathbb C$, $0\in U$, and two 
holomorphic functions $g,h:U\rightarrow \mathbb C$ such that $f(z)=\frac{g(z)}{h(z)}$ for all $z\in U\setminus \{0\}$. 
It is well known that $f$ has a Laurent expansion 
$$f(z)=\sum_{n=\ell}^{+\infty}a_n z^n,\;0<|z|<R, \text{ where }R=\frac{1}{\limsup \sqrt[n]{|a_n|}} >0 \text{ and }\ell\in\mathbb Z.$$
If $a_{\ell}\neq 0$, the number $\ord_{z=0}f(z):=\ell$ is called the \emph{order of zero} of $f$ at $0$. If $\ell\geq 0$, then $f$ is holomorphic
at $0$ and has a zero of order $\ell$ at $0$. If $\ell<0$, then $0$ is a pole of order $-\ell$ of $f$. As in the holomorphic case, we
define $\mathcal M_{0}$ the ring of germs of meromorphic function at $0$. We have that $\mathcal M_0$ is the quotient field of $\mathcal O_0$ and hence
$$\mathcal M_0 \cong Q(\mathbb C\{z\}) = \mathbb C\{z\}[z^{-1}] \cong \frac{\mathbb C\{z\}[t]}{(1-zt)}.$$
In order to simplify the notation, we denote by $f$ a holomorphic (meromorphic) function at $0$ and its germ.

We fix a field $\mathbb F$ such that $\mathbb C \subset \mathbb F \subsetneq \mathcal M_0$ and some germs $f_1,\ldots,f_r\in\mathcal M_0$.
We consider the $\mathbb F$-algebras $S:=\mathbb F[f_1,\ldots,f_r]$ and $\overline S:=\mathbb F[f_1^{\pm 1},\ldots,f_r^{\pm 1}]$.
Our aim is to study the $\mathbb F$-subalgebras 
\begin{align*}
& S^{hol}:=\mathbb F[ f^{\mathbf a}:=f_1^{a_1}\cdots f_r^{a_r}\;:\; \ord_{z=0}f(z) \geq 0,\;\mathbf a=(a_1,\ldots,a_r)\in\mathbb N^r] \subset S. \\
& \overline{S}^{hol}:=\mathbb F[ f^{\mathbf a}:=f_1^{a_1}\cdots f_r^{a_r}\;:\; \ord_{z=0}f(z) \geq 0,\;\mathbf a=(a_1,\ldots,a_r)\in\mathbb Z^r] \subset \overline S.
\end{align*}
In the second section, we present the general properties of $S^{hol}$ and $\overline{S}^{hol}$, using the methods from \cite{mir}.
Theorem $2.3$, Theorem $2.4$ and Theorem $2.5$ are simple generalizations of the main results from \cite{mir}, hence we omit the proofs.

In the third section, we present our main results of the paper. We let $\mathbb F:=\mathcal M_{<1}$ be
the field of meromorphic functions of order $<1$ and we let $f_1,\ldots,f_r$ be some meromorphic functions of finite order with finite number of
zeros and poles. In Proposition $3.1$ we prove that such functions are of the form $R(z)e^{P(z)}$, where $R(z)$ is a rational function and $P(z)$ is a polynomial.
In Theorem $3.3$ we prove that if  $P_1,P_2,\ldots,P_r\in \mathbb C[z]$ are polynomials such that $P_j-P_k$ are non-constant
for all $j\neq k$, then the functions $f_j(z)=e^{P_j(z)}$, $1\leq j\leq r$, are linearly independent over $\mathbb F$.
Moreover, if $d_j:=\deg(P_j)\geq 1$ for $1 \leq j\leq r$ and $d_j\neq d_k$ for all $j\neq k$, then $f_1,\ldots,f_r$ are algebraically independent over $\mathbb F$.
In Corollary $3.4$ we prove similar conclusions, when we replace $f_j$'s with linear combinations $h_j=\sum_{k=1}^r g_{jk}f_k$, $1\leq j\leq r$, where $g_{jk}\in\mathbb F$ 
and the determinant $\det((g_{jk})_{j,k})$ is nonzero. In Corollary $3.5$ we prove that if $\varphi\in\mathbb F$ and $f_j(z)=e^{P_j(z)}$, $1\leq j\leq r$, are as in the hypothesis
of Theorem $3.3$, then $\varphi,f_1,\ldots,f_r$ are linearly (algebraically) independent over $\mathbb C$.
We conclude our paper with Example $3.6$.

\newpage
\section{Preliminaries}

Let $\mathcal M_0$ be the field of germs of meromorphic functions at $0$. Let $\mathbb F$ be a field such that 
$\mathbb C \subset \mathbb F \subsetneq \mathcal M_0$ and let $f_1,\ldots,f_r\in\mathcal M_0$.
Let $S:=\mathbb F[f_1,\ldots,f_r]$ and $\overline S:=\mathbb F[f_1^{\pm 1},\ldots,f_r^{\pm 1}]$. Since $S$ is a domain, we have 
\begin{equation}\label{112}
S\cong \frac{\mathbb F[x_1,\ldots,x_r]}{\mathfrak p}, 
\end{equation}
where $\mathfrak p \subset \mathbb F[x_1,\ldots,x_r]$ is a prime ideal.
Similarly, 
\begin{equation}\label{1112}
\overline S \cong \frac{\mathbb F[x_1^{\pm 1},\ldots,x_r^{\pm 1}]}{\mathfrak q} \cong \frac{\mathbb F[x_1,\ldots,x_r,y_1,\ldots,y_r]}{ \overline{\mathfrak p}}, 
\end{equation}
where $\mathfrak q\subset \mathbb F[x_1^{\pm 1},\ldots,x_r^{\pm 1}]$ is a prime ideal and $\overline{\mathfrak p}\subset \mathbb F[x_1,\ldots,x_r,y_1,\ldots,y_r]$ is a prime ideal such that
$$\mathfrak q \cong \frac{\overline{\mathfrak p}}{(x_1y_1-1,\ldots,x_ry_r-1)}.$$
We consider the $\mathbb F$-subalgebras
\begin{equation}\label{12}
S^{hol}:=\mathbb F[ f^{\mathbf a}:=f_1^{a_1}\cdots f_r^{a_r}\;:\;\mathbf a=(a_1,\ldots,a_r)\in\mathbb N^r\text{ and }\ord_{z=0}f^{\mathbf a}(z)\geq 0] \subset S,
\end{equation}
\begin{equation}\label{122}
\overline{S}^{hol}:=\mathbb F[ f^{\mathbf a}:=f_1^{a_1}\cdots f_r^{a_r}\;:\;\mathbf a=(a_1,\ldots,a_r)\in\mathbb Z^r\text{ and }\ord_{z=0}f^{\mathbf a}(z) \geq 0] \subset \overline S.
\end{equation}
Let $\ell_j:=\ord_{z=0}f_j(z)$, for $1\leq j\leq r$. From \eqref{12} it follows that
\begin{equation}\label{13}
S^{hol}:=\mathbb F[f^{\mathbf a}\;:\; a_1\ell_1+\cdots +a_r\ell_r\geq 0,\;\mathbf a\in\mathbb N^r]. 
\end{equation}
Similarly, from \eqref{122} it follows that
\begin{equation}\label{133}
\overline{S}^{hol}:=\mathbb F[f^{\mathbf a}\;:\; a_1\ell_1+\cdots +a_r\ell_r\geq 0,\;\mathbf a\in\mathbb Z^r]. 
\end{equation}
We consider the semigroups
\begin{equation}\label{14}
H:=\{\mathbf a=(a_1,\ldots,a_r)\;:\;a_1\ell_1+\cdots+a_r\ell_r \geq 0\}\subset \mathbb N^r,
\end{equation}
\begin{equation}\label{144}
\overline H:=\{\mathbf a=(a_1,\ldots,a_r)\;:\;a_1\ell_1+\cdots+a_r\ell_r \geq 0\}\subset \mathbb Z^r,
\end{equation}
and their associated toric ring 
\begin{equation}\label{15}
\mathbb F[H]:=\mathbb F[x^{\mathbf a}=x_1^{a_1}\cdots x_r^{a_r}\;:\;\mathbf a\in H] \subset \mathbb F[x_1,\ldots,x_r].
\end{equation}
\begin{equation}\label{155}
\mathbb F[\overline H]:=\mathbb F[x^{\mathbf a}=x_1^{a_1}\cdots x_r^{a_r}\;:\;\mathbf a\in \overline H] \subset \mathbb F[x_1^{\pm 1},\ldots,x_r^{\pm 1}].
\end{equation}
We consider the semigroup
\begin{equation}
 \widetilde H:=\{(\mathbf b,\mathbf c)\in \mathbb N^r\times \mathbb N^r\;:\;\mathbf b-\mathbf c\in\overline H\}.
\end{equation}
with its associated toric ring $\mathbb F[\widetilde H]$. One can easily check that 
\begin{equation}\label{cocoo}
 \mathbb F[\overline H] \cong \frac{\mathbb F[\widetilde H]}{(x_1y_1-1,\ldots,x_ry_r-1)}.
\end{equation}
From \eqref{112}, \eqref{13} and \eqref{15} it follows that
\begin{equation}\label{17}
S^{hol} \cong \frac{\mathfrak p + \mathbb F[H]}{\mathfrak p} \cong \frac{\mathbb F[H]}{\mathfrak p \cap \mathbb F[H]}.
\end{equation}
From \eqref{1112}, \eqref{133}, \eqref{155} and \eqref{cocoo} it follows that
\begin{equation}\label{177}
\overline S^{hol} \cong \frac{\mathfrak q + \mathbb F[\overline H]}{\mathfrak q} \cong \frac{\overline{\mathfrak p} + \mathbb F[\widetilde H]}{\overline{\mathfrak p}}
\cong \frac{\mathbb F[\widetilde H]}{\overline{\mathfrak p} \cap \mathbb F[\widetilde H]}.
\end{equation}
There are three cases to consider:
\begin{align*}
 &(i)\; \ell_1>0,\ldots, \ell_p>0, \ell_{p+1} = \cdots= \ell_r = 0, \text{ where }p\geq 0 \\
 &(ii)\; \ell_1<0,\ldots,\ell_q<0,\;\ell_{q+1}=\cdots=\ell_r=0,\text{ where }q\geq 1\\
 &(iii)\; \ell_1>0,\ldots,\ell_p>0,\; \ell_{p+1}<0,\ldots,\ell_q<0,\; \ell_{q+1}=\cdots=\ell_r=0,\;1\leq p < q \leq r.
\end{align*}
In the case $(i)$, we have that 
$$\mathbb F[H]=\mathbb F[x_1,\ldots,x_r] \text{ and } \mathbb F[\overline H]=\mathbb F[x_1,\ldots,x_p,x_{p+1}^{\pm 1},\ldots,x_r^{\pm 1}].$$
In the case $(ii)$, we have that 
$$\mathbb F[H]=\mathbb F[x_{q+1},\ldots,x_r] \text{ and }\mathbb F[\overline H]=\mathbb F[x_1^{-1},\ldots,x_q^{-1},x_{q+1}^{\pm 1},\ldots,x_r^{\pm 1}].$$
Assume we are in the case $(iii)$. Let $v_1,\ldots,v_m \in \mathbb F[x_1,\ldots,x_r]$ be the minimal monomial set of generators of the $\mathbb F$-algebra $\mathbb F[H]$.
In \cite[Proposition 1.3(1)]{mir} we proved that $m\geq r$. We consider the natural epimorphism
\begin{equation}\label{19}
\Phi:\mathbb F[t_1,\ldots,t_m] \to \mathbb F[H],\;\Phi(t_j):=v_j,\;1\leq j\leq m. 
\end{equation}
$I_H:=\Ker(\Phi)$ is called the \emph{toric ideal} of $H$, see \cite{villa} for further details. From \eqref{17} and \eqref{19} it follows that
\begin{equation}\label{110}
S^{hol} \cong \frac{\mathbb F[t_1,\ldots,t_m]}{\Phi^{-1}(\mathfrak p \cap \mathbb F[H])}.
\end{equation}
Now, assume that $w_1,\ldots,w_s$ are the minimal monomial generators of the $\mathbb F$-algebra $\mathbb F[\widetilde H]$. We consider the natural epimorphism
\begin{equation}\label{1999}
 \widetilde \Phi:\mathbb F[t_1,\ldots,t_s] \to \mathbb F[\widetilde H],\; \widetilde \Phi(t_j)=w_j,\;1\leq j\leq s.
\end{equation}
The ideal $I_{\widetilde H}:=\Ker(\widetilde \Phi)$ is the toric ideal of $\widetilde H$. From \eqref{177} and \eqref{1999} it follows that
\begin{equation}\label{1100}
\overline S^{hol} \cong \frac{\mathbb F[t_1,\ldots,t_s]}{\widetilde \Phi^{-1}(\overline{\mathfrak p} \cap \mathbb F[\widetilde H])}.
\end{equation}
\begin{obs}\emph{
Let $K/\mathbb Q$ be a finite Galois extension. 
For the character $\chi$ of the Galois group $G:=Gal(K/\mathbb Q)$
on a finite dimensional complex vector space, let $L(s,\chi):=L(s,\chi,K/\mathbb Q)$ be the corresponding Artin L-function (\cite[P.296]{artin2}). 
Artin conjectured that $L(s,\chi)$ is holomorphic in $\mathbb C\setminus \{1\}$ and $s=1$ is a simple pole. Brauer \cite{brauer}
proved that $L(s,\chi)$ is meromorphic in $\mathbb C$, of order $1$.
Let $\chi_1,\chi_2,\ldots,\chi_r$ be the irreducible characters of $G$. 
Let $f_1:=L(s,\chi_1),\ldots,f_r:=L(s,\chi_r)$.}

\emph{
Artin \cite[Satz 5, P. 106]{artin1} proved that $f_1,\ldots,f_r$ are multiplicatively independent.
F.\ Nicolae proved in \cite{lucrare} that $f_1,\ldots,f_r$ are algebraically independent over $\mathbb C$. 
This result was extended in \cite{forum}
to the field $\mathcal M_{<1}$ of meromorphic functions of order $<1$. 
Let $\mathbb F$ be a field such that $\mathbb C\subset \mathbb F\subset \mathcal M_{<1}$. 
We consider $S:=\mathbb F[f_1,\ldots,f_r]$, $\overline S:=\mathbb F[f_1^{\pm 1},\ldots,f_r^{\pm 1}]$,
 $S^{hol}$, $\overline S^{hol}$, $H$ and $\overline H$ as above.
An extensive study of the semigroup rings $\mathbb F[H]\cong S^{hol}$ and $\mathbb F[\overline H]\cong \overline S^{hol}$ was done in \cite{mir}, 
in the frame of Artin L-functions.}
\end{obs}

We recall the several results from \cite{mir}, which hold in our (more general) context.

\begin{teor}(\cite[Proposition 1.3(2), Theorem 1.4, Proposition 2.2]{mir})
 In the case $(iii)$, the following are equivalent:
\begin{enumerate}
 \item[(1)] $I_H=(0)$.
 \item[(2)] $\mathbb F[H]$ is minimally generated by $r$ monomials.
 \item[(3)] $\ell_1>0$, $\ell_2<0,\ldots, \ell_q<0$, $\ell_{q+1}=\cdots=\ell_r=0$, where $q\geq 2$, and $\ell_1|\ell_j$ for $2\leq j\leq q$.
 \item[(4)] $\mathbb F[H]=\mathbb F[x_1,x_1^{m_2}x_2,\ldots,x_1^{m_q}x_q,x_{q+1},\ldots,x_r]$, where $m_j=-\frac{\ell_j}{\ell_1}$, $2\leq j\leq q$.
 \item[(5)] $\mathbb F[\overline H] = \mathbb F[x_1,(x_1^{m_2}x_2)^{\pm 1},\ldots,(x_1^{m_q}x_q)^{\pm 1},x_{q+1}^{\pm 1},\ldots,x_r^{\pm 1}]$, where $m_j=-\frac{\ell_j}{\ell_1}$, $2\leq j\leq q$.
\end{enumerate}
\end{teor}

Given a monomial $v\in \mathbb F[x_1,\ldots,x_r]$, the support of $v$ is the set $\supp(v)=\{x_j\;:\;x_j|v\}$.
For $1\leq t\leq r-1$, we consider the numbers: 
$$L_t = |\{ \supp(v) \;:\;v\in \mathbb F[H],\;|\supp(v)|=t\}| \text{ and } N_t=\binom{r}{t}-\binom{r-2}{t-1}+1.$$
\begin{teor}(\cite[Theorem 1.6]{mir})
Except the case $(ii)$, the following are equivalent:
\begin{enumerate}
 \item[(1)] $\mathbb F[H]=\mathbb F[x_1,\ldots,x_r]$.
 \item[(2)] $I_H=(0)$ and there exists $1\leq t\leq r-1$ such that $L_t\geq N_t$.
\end{enumerate}
\end{teor}

\begin{teor}(\cite[Theorem 1.13, Proposition 2.3]{mir})
In the case (iii), if 
$\ell_1=\cdots=\ell_p=1,\; \ell_{p+1}=\cdots=\ell_q=-1 \text{ and } \ell_{q+1}=\cdots=\ell_r =0$,
then we have:
\begin{enumerate}
 \item[(1)] $\mathbb F[H]=\mathbb F[x_1,\ldots,x_p,x_{q+1},\ldots,x_r, x_jx_k \;:\; 1\leq j\leq p,\;p+1\leq k\leq q]$.
 \item[(2)] $\mathbb F[\overline{H}] = \mathbb F[x_1,\ldots,x_{p},x_{q+1}^{\pm 1},\ldots,x_r^{\pm 1}, (x_{j}x_k)^{\pm 1} : 1\leq j\leq p,\;p+1 \leq k\leq q].$
 \item[(3)] Letting $\Phi:\mathbb F[t_1,\ldots,t_p,t_{q+1},\ldots,t_r,t_{jk}\;:\;1\leq j\leq p,\;p+1\leq k\leq q]\to \mathbb F[H]$, $\Phi(t_j)=x_j$, $\Phi(t_{jk})=x_jx_k$,
       we have: $$I_H=\Ker(\Phi)=(t_jt_{ik}-t_it_{jk}, t_{jk}t_{im} - t_{jm}t_{ik}\;:\; 1\leq j,i\leq p,\;p+1\leq k,m\leq q).$$
\end{enumerate}
\end{teor}

\section{Main results}

We denote $\mathcal O$ the domain of entire functions. We have that 
$$\mathcal O = \{f(z)=\sum_{n=0}^{+\infty}a_nz^n\;:\;a_n\in\mathbb C,\; \lim_n \sqrt[n]{|a_n|}=0\} \subset \mathcal O_0 \cong \mathbb C\{z\}.$$
Let $f\in \mathcal O$. If there exist a positive number $\rho$ and
constants $A,B > 0$ such that
\begin{equation}\label{ord}
|f(z)|\leq Ae^{B|z|^{\rho}} \text{ for all }z\in\mathbb C, 
\end{equation}
then we say that f has an \emph{order of growth} $\leq\rho$. We define the \emph{order of growth} of $f$ as
$$\rho(f)=\inf\{\rho>0\;:\;f \text{ has an order of growth }\leq\rho \}.$$
For each integer $k \geq 0$ we define canonical factors by
$$E_0(z)=1-z \text{ and } E_k(z)=(1-z)e^{z+\frac{z^2}{2}+\cdots+\frac{z^{k}}{k}}\text{ for }k\geq 1.$$
Let $f\in\mathcal O$ be an entire function with the order of growth $\rho$. From Hadamard's Theorem (see for instance \cite[Theorem 5.1]{stein}), it follows that
\begin{equation}\label{2_2}
 f(z)=z^me^{P(z)}\prod_{n=1}^{\infty}E_k(\frac{z}{z_n}),
\end{equation}
where $k=\lfloor \rho \rfloor$, $z_1,z_2,\ldots$ are the non-zero zeros of $f$, $P$ is a polynomial of degree $\leq k$ and $m$ is 
the order of the zero of $f$ at $z=0$. In particular, if the number of zeros of $f$ is finite, then 
\begin{equation}\label{2_3}
 \rho=k \text{ and }f(z)=Q(z)e^{P(z)}, \text{ where }Q\in \mathbb C[z].
\end{equation}
It is well known that the field of meromorphic functions on $\mathbb C$, denoted by $\mathcal M$ is the quotient field of $\mathcal O$.
Moreover, if $f$ is meromorphic with order of growth $\leq \rho$, then $f$ is the quotient of two holomorphic functions with order of growth $\leq \rho$.
For any $\rho>0$, we denote $\mathcal O_{<\rho}$ the domain of entire functions with order of growth $<\rho$, and
 $\mathcal M_{<\rho}$ the quotient field of $\mathcal O_{<\rho}$, that is the field of meromorphic functions of order $<\rho$.

\begin{prop}
 If $f$ is a meromorphic function with order of growth $\rho$ with finitely many zeros and poles, then $\rho$ is an integer and $f(z)=R(z)e^{P(z)}$,
where $R(z)\in\mathbb C(z)$ is a rational function, and $P\in \mathbb C[z]$ is a polynomial of degree $\rho$.
\end{prop}

\begin{proof}
Since $f$ is meromorphic of order $\rho$, we can write $f(z)=\frac{g(z)}{h(z)}$, where $g$ and $h$ are holomorphic of order $\leq \rho$
and at least one of then has the order of growth $\rho$. From \eqref{2_3} it follows that $\rho$ is integer and
$$g(z)=Q_1(z)e^{P_1(z)} \text{ and } h(z)=Q_2(z)e^{P_2(z)},$$
where $P_1,Q_1,P_2$ and $Q_2$ are polynomials with $\max\{\deg(P_1),\deg(P_2)\}=\rho$. Therefore
$$f(z)=\frac{Q_1(z)}{Q_2(z)}e^{P_1(z)-P_2(z)}.$$
Since $f$ has the order of growth $\rho$, it follows that $\deg(P_1-P_2)=\rho$, as required.
\end{proof}

\begin{obs}
Let $f_1,\ldots,f_r$ be some meromorphic functions with finite orders of growth $\rho_1,\ldots,\rho_r$ and 
finitely many zeros and poles. From Proposition $3.1$ it follows that
$$f_j(z)=R_j(z)e^{P_j(z)} \text{ for }1\leq j\leq r,$$
where $R_j\in\mathbb C(z)$ and $P_j\in \mathbb C[z]$ for $1\leq j\leq r$. We have the $\mathbb C(z)$-algebra isomorphisms
\begin{align*}
& \mathbb C(z)[f_1,\ldots,f_r] \cong \mathbb C(z)[e^{P_1(z)},\ldots,e^{P_r(z)}], \\
& \mathbb C(z)[f_1^{\pm 1},\ldots,f_r^{\pm 1}] \cong \mathbb C(z)[e^{\pm P_1(z)},\ldots,e^{\pm P_r(z)}].
\end{align*}
Since $\mathbb C(z)$ is a subfield of $\mathcal M_{<1}$, it follws that we have the $\mathcal M_{<1}$-algebra isomorphism
\begin{align*}
& S:=\mathcal M_{<1}[f_1,\ldots,f_r] \cong \mathcal M_{<1}[e^{P_1(z)},\ldots,e^{P_r(z)}], \\
& \overline S:=\mathcal M_{<1}[f_1^{\pm 1},\ldots,f_r^{\pm 1}] \cong \mathcal M_{<1}[e^{\pm P_1(z)},\ldots,e^{\pm P_r(z)}].
\end{align*}
However, in general $S^{hol} \subsetneq (\mathcal M_{<1}[e^{P_1(z)},\ldots,e^{P_r(z)}])^{hol}=\mathcal M_{<1}[e^{P_1(z)},\ldots,e^{P_r(z)}]$,
as the functions $f_1,\ldots,f_r$ could have poles at $z=0$.
\end{obs}

In the following theorem, we give a criterion for the linear (algebraic) independence of the functions $f_j(z)=e^{P_j(z)}$, $1\leq j\leq r$,
over the field $\mathcal M_{<1}$.

\begin{teor}
Let $P_1,P_2,\ldots,P_r\in \mathbb C[z]$ be polynomials such that $P_j-P_k$ is non-constant
for any $j\neq k$. Let $d_j:=\deg(P_j)$ for $1 \leq j\leq r$ and 
$$f_j:\mathbb C \to \mathbb C^*, \; f_j(z):=e^{P_j(z)}, \;(\forall) z\in \mathbb C,\; 1\leq j\leq r.$$ 
\begin{enumerate}
 \item[(1)] The holomorphic functions $f_1,\ldots,f_r$ are linearly independent over $\mathcal M_{<1}$. 
 \item[(2)] If $d_j\geq 1$ for all $1\leq j\leq r$ and $|\{d_1,\ldots,d_r\}|=r$ then $f_1,\ldots,f_r$ are algebraically independent over $\mathcal M_{<1}$.
\end{enumerate}
\end{teor}

\begin{proof}
(1) Note that $f_j$ is an entire functions of order $d_j:=\deg(P_j)$, for any $1\leq j\leq r$.
We use induction on $r\geq 1$. The case $r=1$ is obvious. Assume $r\geq 2$ and
let $g_1,\ldots,g_r\in\mathcal M_{<1}$ such that $$g_1f_1+g_2f_2+\cdots+g_rf_r=0.$$ 
If $g_r=0$, then we are done by induction hypothesis. Without any loss of generality, we can assume that $g_r$ is identically $1$.
It follows that
\begin{equation}\label{22}
 g_1(z) e^{P_1(z)-P_r(z)} + \cdots + g_{r-1}(z) e^{P_{r-1}(z)-P_{r}(z)} + 1 = 0,\text{ for all }z\in\mathbb C.
\end{equation}
Differentiating \eqref{22} it follows that 
\begin{eqnarray*}
& (g'_1(z)+g_1(z)(P_1(z)-P_r(z))) e^{P_1(z)-P_r(z)} + \cdots  \\
& + (g'_{r-1}(z)+g_{r-1}(z)(P_{r-1}(z)-P_r(z)))e^{P_{r-1}(z)-P_{r}(z)}= 0,\text{ for all }z\in\mathbb C.
\end{eqnarray*}
Since $(P_j-P_r)-(P_k-P_r)=P_j-P_k$ are non-constant for all $1\leq j\neq k\leq r-1$, 
by induction hypothesis, it follows that
$g'_j+(P_j-P_r)g_j = 0$, for all $1\leq j\leq r-1$, hence
\begin{equation}\label{23}
 g_j(z) = C_je^{P_r-P_j}(z),\text{ where }C_j\in\mathbb C,\;1\leq j\leq r.
\end{equation}
If $C_j\neq 0$, since $\deg(P_r-P_j)\geq 1$, from \eqref{23} it follows that $g_j$ is a holomorphic function of order $\geq 1$, 
a contradiction. Hence $g_j=0$ for all $1\leq j\leq r-1$ and thus we get a contradiction from \eqref{22}.

(2) Let $Q\in \mathbb F[t_1,\ldots,t_r]$ be a polynomial such that $Q(f_1,\ldots,f_r)=0$. We have that
     $$Q(t_1,\ldots,t_n)=\sum_{\mathbf a\in \mathbb N^r}g_{\mathbf a}t^{\mathbf a}, \text{ where }t^{\mathbf a}:=t_1^{a_1}\cdots t_r^{a_r},\;g_{\mathbf a}\in \mathbb F,$$
     and only a finite number of $g_{\mathbf a}$'s are nonzero. Hence 
     \begin{equation}\label{24}
       Q(f_1,\ldots,f_n)(z)=\sum_{\mathbf a\in \mathbb N^r}g_{\mathbf a}(z)f^{\mathbf a}(z),\text{ where } f^{\mathbf a}:=f_1^{a_1}\cdots f_r^{a_r}.
     \end{equation}
     For any $\mathbf a\in\mathbb N^r$, we have
     \begin{equation}\label{25}
       f^{\mathbf a}(z) = e^{P_{\mathbf a}(z)},\text{ where }P_{\mathbf a}:=a_1P_1+\cdots+a_rP_r.
     \end{equation}
     Let $\mathbf b\neq \mathbf a \in \mathbb N^r$. Since the $d_j$'s are pairwise disjoint, the polynomial
     \begin{equation}\label{26}
       P_{\mathbf a} - P_{\mathbf b} = (a_1-b_1)P_1+\cdots+(a_r-b_r)P_r,
     \end{equation}
     in non-constant. From  \eqref{25}, \eqref{26} and (i) it follows that the set 
     $\{f^{\mathbf a}\;:\;\mathbf a\in \mathbb N^r\}$ is linearly independent over $\mathbb F$. Hence, from \eqref{24}, 
     we get $Q=0$, as required.
\end{proof}

\begin{cor}
Let $g_{jk}\in \mathcal M_{<1}$ for all $1\leq j,k\leq r$ such that
$$D(z):=\det(g_{jk}(z))_{j,k}\neq 0, (\forall)z\in A,$$
where $A\subset \mathbb C$ is a non-discrete subset. In the hypothesis of Theorem $3.3$, the 
meromorphic functions $h_j:=\sum_{k=1}^r g_{jk}f_k$, $1\leq j\leq k$, are linearly independent over $\mathcal M_{<1}$.

Moreover, in the hypothesis $(2)$ of Theorem $3.3$, the functions $h_1,\ldots,h_r$ are 
algebraically independent over $\mathcal M_{<1}$.
\end{cor}

\begin{proof}
As $D(z)$ is non constant on a non-discrete subset $A\subset \mathbb C$, it follows 
that $D \in \mathcal M_{<1}$ is nonzero. Since, from Theorem $3.3$, $\{f_1,\ldots,f_r\}$ are linearly independent over $\mathcal M_{<1}$,
it follows that $\{h_1,\ldots,h_r\}$ are also linearly independent over $\mathcal M_{<1}$.

Since $D$ is nonzero and $f_1,\ldots,f_r$ are algebraically independent over $\mathbb F$, it follows that the map
$$\mathcal M_{<1}[f_1,\ldots,f_r] \to \mathcal M_{<1}[h_1,\ldots,h_r],\;f_j\mapsto h_j,\;1\leq j\leq r,$$
is a $\mathcal M_{<1}$-algebra isomorphism, hence $h_1,\ldots,h_r$ are algebraically independent.
\end{proof}

\begin{cor}
Let $\varphi\in \mathcal M_{<1}$ be a a non-constant function and let $P_1,P_2,\ldots,P_r\in \mathbb C[z]$ 
be non-constant polynomials of degrees $d_1,d_2,\ldots,d_r$ such that $P_j-P_k$ is non-constant for any $j\neq k$. Let $f_j(z)=e^{P_j(z)}$ 
for $1\leq j\leq r$. Then:
\begin{enumerate}
 \item[(1)] $\varphi,f_1,\ldots,f_r$ are linearly independent over $\mathbb C$.
 \item[(2)] If $|\{d_1,\ldots,d_r\}|=r$ then $\varphi,f_1,\ldots,f_r$ are algebraically independent over $\mathbb C$.
\end{enumerate}
Moreover, if $A=(a_{ij})_{0\leq i,j\leq r}$ is a nonsingular matrix with entries in $\mathbb C$, and 
$$g_j = a_{0j}\varphi + a_{1j}f_1 + \cdots + a_{rj}f_j \text{ for }0\leq j\leq r,$$
then the conclusions $(1)$ and $(2)$ holds if we replace $\varphi,f_1,\ldots,f_r$ with $g_0,g_1,\ldots,g_r$.
\end{cor}

\begin{proof}
(1) We consider a linear combination 
\begin{equation}\label{cucu}
a\varphi+a_1f_1+\cdots+a_rf_r=0,\;a\in\mathbb C\text{ and }a_j\in \mathbb C\text{ for }1\leq j\leq r. 
\end{equation}
If $a=0$, then $a_1=\cdots=a_r=0$ from Theorem $3.3(1)$. Assume $a\neq 0$.
Note that, at most one of the polynomials $P_j$'s is constant. We consider two cases: 

(i) If $d_1=0$ and $d_j\geq 1$ for $2\leq j\leq r$, then $f_1\in \mathbb C$.
       Let $$g_j(z):=\frac{1}{a\varphi(z)+a_1f_1},\text{ for }2\leq j\leq r.$$
       From \eqref{cucu} it follows that $$1+g_2f_2+\cdots+g_rf_r=0.$$
       According to the proof of Theorem $3.3(1)$, this yields a contradiction.
 
(ii) If $d_j\geq 1$ for $1\leq j\leq r$, then we let 
       $$g_j(z):=\frac{1}{a\varphi(z)}\text{ for }1\leq j\leq r.$$
       From \eqref{cucu} it follows that $$1+g_1f_1+\cdots+g_rf_r=0,$$
       which, according to the proof of Theorem $3.3(1)$, yields a contradiction.

(2) From Theorem $3.3(2)$ it follows that $f_1,\ldots,f_r$ are algebraically independent over $\mathcal M_{<0}$
    and hence over $\mathbb C(\varphi)$. On the other hand, the nonconstant function $\varphi$ is algebraically 
    independent over $\mathbb C$. Thus $\varphi,f_1,\ldots,f_r$ are algebraically independent over $\mathbb C$.

The last assertion follows from Corollary $3.4$.
\end{proof}

We conclude our paper with a list of examples.

\begin{exm}\emph{
(1) Let $f_1(z)=z, f_2(z)=\frac{e^{-z}}{z^2}, f_3(z)=e^z$ and $S=\mathbb C[f_1,f_2,f_3]$. According to Corollary $3.5$, 
the meromorphic functions $f_1,f_2,f_3$ are linearly independent over $\mathbb C$. One can easily check that 
$$S:=\mathbb C[f_1,f_2,f_3] \cong \frac{\mathbb C[x_1,x_2,x_3]}{\mathfrak p},\text{ where }\mathfrak p =(x_1^2x_2x_3-1).$$
On the other hand, we have that
\begin{align*}
& \overline S:=\mathbb C[f_1^{\pm 1},f_2^{\pm 1},f_3^{\pm 1}] \cong \frac{\mathbb C[x_1,x_2,x_3,y_1,y_2,y_3]}{\overline{\mathfrak p}},\text{ where } \\
& \overline{\mathfrak p}=(x_1y_1-1,x_2y_2-1,x_3y_3-1,x_1^2x_2x_3-1).
\end{align*}
We consider the semigroups
\begin{align*}
 & H=\{(a_1,a_2,a_3)\in\mathbb N^3\;:\;f_1^{a_1}f_2^{a_2}f_3^{a_3} \text{ is holomorphic at } 0\}.\\
 & \overline H=\{(a_1,a_2,a_3)\in\mathbb Z^3\;:\;f_1^{a_1}f_2^{a_2}f_3^{a_3} \text{ is holomorphic at } 0\}.
\end{align*}
Since $\ell_1=\ord_{z=0}f_1 = 1$, $\ell_2=\ord_{z=0}f_2=-2$ and $\ell_3=\ord_{z=0}f_3=0$, from Theorem $2.2$ it follows that
\begin{align*}
& \mathbb C[H] = \mathbb C[x_1,x_1^2x_2,x_3] \cong \mathbb C[t_1,t_2,t_3],\; I_H=(0) \text{ and }\\
& \mathbb C[\overline H]=\mathbb C[x_1,(x_1^2x_2)^{\pm 1},x_3^{\pm 1}]\cong \mathbb C[t_1,t_2^{\pm 1},t_3^{\pm 1}].
\end{align*}
From \eqref{17} and \eqref{110} it follows that
$$S^{hol} = \mathbb C[z,e^{-z},e^{z}] \cong \frac{\mathbb C[x_1,x_1^2x_2,x_3]}{(x_1^2x_2x_3-1)} \cong \frac{\mathbb C[t_1,t_2,t_3]}{(t_2t_3-1)} \cong \mathbb C[t_1,t_2^{\pm 1}].$$
From \eqref{177} and \eqref{1100} it follows that
$\overline S^{hol} = \mathbb C[z,e^{-z},e^{z}] \cong \mathbb C[t_1,t_2^{\pm 1}]$.}

\emph{
(2) Let $f_1(z)=z, f_2(z)=\sin z, f_3(z)=\frac{e^z}{z},\;f_4(z)=\frac{e^{-z}}{z}$ and $S=\mathbb C[f_1,f_2,f_3,f_4]$. We have
$$S\cong \frac{\mathbb C[x_1,x_2,x_3,x_4]}{(x_1^2x_3x_4-1)}.$$
Since $\ell_1=\ell_2=1$ and $\ell_3=\ell_4-1$, from Theorem $2.3$ it follows that
$$\mathbb C[H] = \mathbb C[x_1,x_2,x_1x_3,x_1x_4,x_2x_3,x_2x_4] \cong \frac{\mathbb C[t_1,t_2,t_{13},t_{14},t_{23},t_{24}]}{(t_1t_{23}-t_2t_{13}, t_1t_{24}-t_2t_{14}, t_{13}t_{24}-t_{14}t_{23})}.$$
From \eqref{110} it follows that \small
$$S^{hol}=\mathbb C[z,\sin z, e^z, e^{-z}, \frac{e^z\sin z}{z},\frac{e^{-z}\sin z}{z}] \cong 
\frac{\mathbb C[t_1,t_2,t_{13},t_{14},t_{23},t_{24}]}{(t_1t_{23}-t_2t_{13}, t_1t_{24}-t_2t_{14}, t_{13}t_{24}-t_{14}t_{23}, t_{13}t_{14}-1)}.$$}
\normalsize

\end{exm}


\vspace{2mm} \noindent {\footnotesize
\begin{minipage}[b]{15cm}
Simion Stoilow Institute of Mathematics\\
Research unit 5,P.O.Box 1-764, Bucharest 014700, Romania and\\
Politehnica University of Bucharest, Faculty of Applied Sciences,\\ 
Department of Mathematical Methods and Models,\\
Bucharest 060042, Romania\\
E-mail: mircea.cimpoeas@imar.ro\\
\end{minipage}}


\begin{thebibliography}{99}




\normalsize
\baselineskip=17pt


 \bibitem{artin1} E.\ Artin, \emph{\"Uber eine neue Art von L-Reihen},  Abh. Math. Sem. Hamburg \textbf{3} (1924), 89--108.
 \bibitem{artin2} E.\ Artin, \emph{Zur Theorie der L-Reihen mit allgemeinen Gruppencharakteren}, Abh. Math. Sem. Hamburg \textbf{8} (1931), 292--306.
 \bibitem{brauer} R.\ Brauer, \emph{On Artin's L-series with general group characters}, Ann.of Math (2) \textbf{48}, (1924), 502--514.
\bibitem{mir} M.\ Cimpoea\c s, \emph{On the semigroup ring of holomorphic Artin L-functions}, to appear in Colloquium Mathematicum (2019).
\bibitem{forum} M.\ Cimpoea\c s, F.\ Nicolae, \emph{Independence of Artin L-functions}, Forum Math. \textbf{31, no. 2} (2019), 529--534.
\bibitem{lucrare} F.\ Nicolae, \emph{On Artins's L-functions. I}, J. reine angew. Math. \textbf{539} (2001), 179--184.
\bibitem{stein} E.\ M.\ Stein, R.\ Shakarchi, \emph{Complex analysis}, Princeton Lectures in Analysis II (2003).
\bibitem{villa} R.\ H.\ Villareal, \emph{Monomial algebras}, Monographs and Textbooks in Pure and Applied
Mathematics, vol. \textbf{238}, Marcel Dekker Inc., New York (2001).
\end{thebibliography}
\end{document}